\documentclass[12pt]{article}

\usepackage{latexsym}
\usepackage{graphics}
\usepackage{xspace}
\usepackage{psfrag}
\usepackage{epsfig}
\usepackage{pst-all}
\usepackage{amssymb, amsmath, amsthm, verbatim, mathrsfs}

\newcommand {\bel}[1]{\begin{align*}}
\newcommand {\eel}[1]{\end{align*}}
\newcommand {\bea}{\begin{eqnarray}}
\newcommand {\eea}{\end{eqnarray}}

\newcommand{\pr}{\mathbb{P}}
\newcommand{\E}{\mathbb{E}}
\newcommand{\R}{\mathbb{R}}

\newcommand{\G}{\mathbb{G}}

\newcommand{\ER}{Erd{\"o}s-R\'{e}nyi }
\newcommand{\cNfloor}{\lfloor cN\rfloor}

\newcommand{\mb}[1]{\mbox{\boldmath $#1$}}

\newcommand{\ignore}[1]{\relax}

\newtheorem{theorem}{Theorem}

\newtheorem{lemma}{Lemma}
\newtheorem{conj}{Conjecture}
\newtheorem{prop}{Proposition}
\newtheorem{coro}{Corollary}

\newtheorem{Defi}{Definition}
\newtheorem{Assumption}{Assumption}

\definecolor{Red}{rgb}{1,0,0}
\definecolor{Blue}{rgb}{0,0,1}
\definecolor{Olive}{rgb}{0.41,0.55,0.13}
\definecolor{Green}{rgb}{0,1,0}
\definecolor{MGreen}{rgb}{0,0.8,0}
\definecolor{DGreen}{rgb}{0,0.55,0}
\definecolor{Yellow}{rgb}{1,1,0}
\definecolor{Cyan}{rgb}{0,1,1}
\definecolor{Magenta}{rgb}{1,0,1}
\definecolor{Orange}{rgb}{1,.5,0}
\definecolor{Violet}{rgb}{.5,0,.5}
\definecolor{Purple}{rgb}{.75,0,.25}
\definecolor{Brown}{rgb}{.75,.5,.25}
\definecolor{Grey}{rgb}{.5,.5,.5}
\definecolor{Pink}{rgb}{1,0,1}
\definecolor{DBrown}{rgb}{.5,.34,.16}
\definecolor{Black}{rgb}{0,0,0}

\author{
 {\sf David Gamarnik }
  \thanks{Operations Research Center and Sloan School of Management, MIT, Cambridge, MA,  02139, e-mail: {\tt
gamarnik@mit.edu}. Research supported by the NSF grants CMMI-1031332.}
}

\begin{document}

\title{Right-convergence of sparse random graphs}

\maketitle

\begin{abstract}
The paper is devoted to the problem of establishing right-convergence of sparse random graphs. This concerns the convergence
of the logarithm of number of homomorphisms from graphs or hyper-graphs $\G_N, N\ge 1$ to some target graph $W$.
The theory of dense graph convergence, including random dense graphs, is now well understood~\cite{BorgsChayesEtAlGraphLimitsI},\cite{BorgsChayesEtAlGraphLimitsII},\cite{LovaszSzegedy},\cite{ChatterjeeVaradhan},
but its counterpart for sparse random graphs presents some fundamental difficulties. Phrased in the statistical physics terminology,
the issue is the existence of the log-partition function limits, also known as free energy limits, appropriately normalized for the
Gibbs distribution associated with $W$. In this paper we prove that the sequence of sparse \ER graphs is right-converging
when the tensor product associated with the target graph $W$ satisfies certain convexity property. We treat the case of discrete and continuous target graphs $W$.
The latter case allows us to prove a special case of Talagrand's recent conjecture (more accurately stated as level III Research Problem
6.7.2 in his recent book~\cite{TalagrandBook}), concerning the existence of the limit
of the measure of a set obtained from $\R^N$ by intersecting it with linearly in $N$
many subsets, generated according to some common probability law.

Our proof is based on the interpolation technique, introduced first by Guerra and Toninelli~\cite{GuerraTon} and
developed further in~\cite{FranzLeone},
\cite{FranzLeoneToninelliRegular}, \cite{PanchenkoTalagrand}, \cite{MontanariLDPCInterpolation},
\cite{BayatiGamarnikTetali}, \cite{AbbeMontanari}, \cite{ContucciDommersGiardinaStarr}.
Specifically, \cite{BayatiGamarnikTetali} establishes the right-convergence property for \ER graphs
for some special cases of $W$. In this paper most of the results in \cite{BayatiGamarnikTetali} follow
as a special case of our main theorem.
\end{abstract}

\section{Introduction}
Given two graphs $\G$ and $W$, a graph homomorphism is a mapping from the nodes of $\G$ to the nodes of $W$, such that
every edge in $\G$ is mapped onto an edge in $W$. When nodes and edges of $W$ are weighted, the homomorphism inherits a certain
weight itself (see Section~\ref{section:Definitions} for details).
A sequence of $N$-node graphs $\G_N, N\ge 1$ is defined to be right-converging with respect to $W$ if the logarithm of the sum of homomorphisms weights,
normalized by $N$, has a limit. In the statistical physics terminology, the nodes of $W$ correspond to spin values, and
the sum of homomorphism weights is called the partition function.

The theory of graph convergence is now well developed
for the case of dense graphs (graphs with number of edges of the order $O(N^2)$),
see~\cite{BorgsChayesEtAlGraphLimitsI},\cite{BorgsChayesEtAlGraphLimitsII},\cite{LovaszSzegedy},\cite{ChatterjeeVaradhan},
where the normalization is appropriately $N^2$, not $N$. The theory of sparse graphs convergence, however,
presents some challenges~\cite{BollobasRiordanMetrics},\cite{BorgsChayesKahnLovasz}, as even establishing some of the basic properties
of convergence for sparse graphs remain conjectures at best. For example, it is an open problem
to show that the most basic sequence of sparse random graphs, namely the sequence of sparse \ER graphs
is right-converging with respect to every target graph $W$.

In this paper we prove a special case of this conjecture, under the assumption that a certain tensor product associated with the
target graph $W$ satisfies some convexity property. Our additional technical assumption is
the existence of spin values (labels of $W$) with positive interaction with every other spin value (namely the corresponding edge weights are positive).
This assumption is adopted in order to avoid potentially nullifying the partition function. See~\cite{AbbeMontanari} where a similar
issue is treated instead by conditioning for partition function to stay positive.
We formulate the problem of right-convergence both for the case of discrete and continuous spin values, the latter corresponding to the case
when the nodes of $W$ are indexed by real values and node and edge weights are random quantities described
by some random measurable functions. Our framework is rich enough to accommodate a recent conjecture by Talagrand regarding
the existence of the limit of a measure of the set obtained from $\R^N$ by intersecting it with subsets
of $\R^N$, chosen i.i.d. according to some common probability law, see Section~\ref{section:ConjecturesResults} for the precise statement. More accurately stated as
Research Problem 6.7.2 in his recent book~\cite{TalagrandBook}, this is one of the many fascinating so-called Level III
problems in the book.

Our method of proof is based on the interpolation technique introduced by Guerra and Toninelli~\cite{GuerraTon}
in the context of Sherrington-Kirkpatrick model in statistical physics and further developed for the case of sparse graphs
(called diluted spin glass models in the statistical physics literature) by Franz and Leone~\cite{FranzLeone},\cite{FranzLeoneToninelliRegular},
Panchenko and Talagrand~\cite{PanchenkoTalagrand}, Montanari~\cite{MontanariLDPCInterpolation}, Bayati et al~\cite{BayatiGamarnikTetali},
Abbe and Montanari~\cite{AbbeMontanari}.
The idea is to build a sequence of graphs interpolating between a random graph on $N$ nodes, on the one hand, and
a disjoint union of two random graphs with $N_1$ and $N_2$ nodes on the other hand, where $N_1+N_2=N$. The interpolation is constructed
in such a way that in every step of the interpolation the log-partition function increases or decreases
in expectation (depending on a model). Such a property means that the expected log-partition function is sub- or super-additive, thus
implying the existence of the limit vis-\'{a}-vi the Fekete's lemma. The convergence with high probability requires an additional
concentration type argument, which for our case turns out to be more involved than usual,  due to the continuity of spin values.
We show that the super-additivity property holds for the class of models of interest  when the assumptions of our main theorem hold.
We further verify the assumptions for a broad scope of models considered earlier in the literature.
In particular, most of the results obtained in~\cite{BayatiGamarnikTetali}, including those regarding Independents Set, Partial Coloring, Ising
 and random K-SAT
models,  follow from the main result of the present paper as a special case. As mentioned above, we do suspect that the right-convergence holds for sparse \ER
graphs for \emph{all} target graphs $W$ and state this explicitly as a conjecture, since, at the present time we do not have a single counterexample,
however contrived.

The remainder of the paper is structured as follows. The notions of graph homomorphisms and right-convergence are introduced
in the following section. Many examples are discussed also in the same section. Finally, the section introduces the definition
of graph homomorphism for continuous spin values. While continuous spin models are ubiquitous in the physics literature, they
are rarely discussed in the context of graph theory. Our model assumptions, conjectures and our main result are stated in Section~\ref{section:ConjecturesResults}.
In the same section we provide an in depth discussion of the convexity of tensor products property - the principal technical tool underlying our main result,
and discuss its relevance for examples introduced earlier in Section~\ref{section:Definitions}. In Section~\ref{section:Concentration} we establish some
basic properties of the log-partition functions and establish a  concentration result. The interpolation technique
is introduced in Section~\ref{section:Interpolation} and the proof of our main result is found in the same section.

We finish this section by introducing  some notations. $\mb{e}_d=(1,1,\ldots,1)^T$ denotes a $d$-dimensional vector of ones.
$\mb{1}(A)$ is an indicator function, taking value $1$ when the event $A$ takes place and zero otherwise.
Throughout the paper we will use standard order of magnitude notations $O(\cdot)$ and $o(\cdots)$ where the constants hidden $O$ should be clear from the context.
$\R~ (\R_+)$ denotes the set of real (non-negative real) values.

\section{Graph homomorphisms and right-convergence}\label{section:Definitions}
Consider a $K$-uniform directed hypergraph $\G$ on nodes $\{1,2,\ldots,N\}\triangleq V(\G)$. Let $E(\G)$ be the set of hyperedges
of $\G$, each hyperedge
$e=(u^e_1,\ldots,u^e_K)\in E(\G)$ being an ordered sets of $K$ nodes in $\G$.
We fix a positive integer $q$ and refer to integers $0,1,\ldots,q-1$ as colors or spin values,
interchangeably. For every node $u\in V(\G)$ of the hypergraph,
let $\mathcal{N}(u,\G)$ be the set of edges incident to $u$. Then $|\mathcal{N}(u,\G)|$ is the degree of $u$ in $\G$,
namely the total number of hyperedges containing $u$. Similarly, for hyperevery edge $e$, let $\mathcal{N}(e,\G)$ be the set of hyperedges incident to $e$,
including $e$. Namely, $\mathcal{N}(e,\G)$ is the set of hyperedges sharing at least one node with $e$.
For simplicity, from this point on we will use the terms
graphs and edges in place of hypergraphs and hyperedges.

Each node $u\in V(\G)$ of $\G$ is associated with a random map $h_u:\{0,1,\ldots,q-1\}\rightarrow \R_+$ called node $u$ potential.
The sequence $h_u, u\in V(\G)$ is assumed to be i.i.d. distributed according to some probability measure $\nu_h$.
Similarly, each hyperedge $e\in E(\G)$
is associated with a random map $J_e:\{0,1,\ldots,q-1\}^K\rightarrow \R_+$ called hyperedge $e$ potential.
The sequence $J_e, e\in E(\G)$ is i.i.d. as well with a common probability measure $\nu_J$. Further details concerning
probability measures $\nu_h$ and $\nu_J$ will be discussed later. For many special cases it will be assume
that these measures are singletons. Namely, the maps $h_u$ and $J_e$ are deterministic.

Given an arbitrary map $\sigma:V(\G)\rightarrow \{0,\ldots,q-1\}$,
we associate with it a (random) weight
\begin{align}\label{eq:weightSigma}
H(\sigma)\triangleq \prod_{u\in V(\G)}h_u(\sigma(u))\prod_{e\in E(\G)} J_{e}(\sigma(u^e_1),\ldots,\sigma(u^e_K)).
\end{align}
Any such map is called \emph{homomorphism} for reasons explained below.
The following random variable is defined to be the partition function of $\G$:
\begin{align}\label{eq:Z}
Z(\G)\triangleq \sum_\sigma H(\sigma),
\end{align}
where the sum is taken over all maps $\sigma:V(\G)\rightarrow \{0,\ldots,q-1\}$. We note that the value $Z(\G)=0$ is possible according to this definition.
In case $Z(\G)>0$,  this induces a random Gibbs probability measure on the set of maps $\sigma$, where the probability mass of $\sigma$ is
$H(\sigma)/Z(\G)$. The graph $\G$ with potentials $h_u,J_e$ and the corresponding Gibbs measure is also commonly called Markov Random Field
in the Electrical Engineering literature.

Already a very rich class of models is obtained when the potentials $h_u$ and $J_e$ are deterministic, denoted by $h$ and $J$ for simplicity.
Consider, for example a special case $h=1$. Further, suppose $K=2$
and $J$ is a symmetric zero-one valued function in its two arguments. Consider an undirected graph $W$ on nodes
$0,\ldots,q-1$, where $(i,j)$ is an edge in $W$ if and only if $J(i,j)=1$ (the symmetry of $J$ makes the definition consistent).
Then observe that $H(\sigma)=1$ if $\sigma$ defines a graph homomorphism from $\G$ to $W$ and $H(\sigma)=0$ otherwise.
(Recall that given two graphs $\G_1,\G_2$,  a map $\sigma:V(\G_1)\rightarrow V(\G_2)$ is called graph homomorphism
if for every $(i_1,i_2)\in E(\G_1)$ we have  $(\sigma(i_1),\sigma(i_2))\in E(\G_2)$).
Thus $Z(\G)$ is the number of homomorphisms from $\G$ to $W$. In that sense we can think of  an arbitrary map $\sigma:V(\G)\rightarrow \{0,1,\ldots,q-1\}$
as a graph homomorphism and the associated partition function $Z(\G)$ is the "number" of homomorphisms.

Let us discuss some well-known examples from combinatorics  and statistical physics in the context of our definition.

\subsection{Examples}\label{subsection:Models}
\textbf{Independent Sets (Hard-Core) model.}
A parameter $\lambda>0$ is fixed. $K=2,q=2$. The node and edge potentials are deterministic denoted by $h$ and $J$ respectively. In particular,
we set $h(1)=\lambda$ and $h(0)=1$.
The edge potential is defined by
 $J(i_1,i_2)=0$
if $i_1=i_2=1$ and $J(i_1,i_2)=1, ~0\le i_1,i_2\le 1$, otherwise.  Then $H(\sigma)>0$ if  the set of nodes $u$ with $\sigma(u)=1$
is an independent set in $\G$, and $H(\sigma)=0$ otherwise.  Moreover, in the former case $H(\sigma)=\lambda^{k}$, where $k$ is the cardinality of the
corresponding independent set. Thus $Z(\G)$ is the usual partition function associated with the independent set model, also known as the hard-core gas model.
The parameter $\lambda$ is commonly called
\emph{activity} or \emph{fugacity}. When $\lambda=1$, $Z(\G)$ is simply the number of independent sets in the graph $\G$.

\vspace{.2in}
\noindent
\textbf{Partial Colorings (Potts) model.} $K=2$, $h_u=h\equiv 1$. A parameter $\beta\ge 0$ is fixed. The edge potentials are identical
 for all edges and given by $J(i_1,i_2)=1$ if $i_1\ne i_2$ and $=\exp(-\beta)$ otherwise.
In this case $H(\sigma)=\exp(-\beta k)$, where $k$ is the number of monochromatic edges, namely edges receiving the same color at the incident vertices.
The case $\beta>0 ~(\beta<0)$ is usually called anti-ferromagnetic (ferromagnetic) Potts model. In this paper we focus on the anti-ferromagnetic case.

\vspace{.2in}
\noindent
\textbf{Ising model.}
$K=2,q=2$. $h_u(1)=h(1)=h$, for some constant value $h\in\R$, and $h_u(0)=h(0)=1$. The edge potentials are identical for all edges and given by
$J(i_1,i_2)=\exp(-\beta \mb{1}(i_1\ne i_2)+\beta \mb{1}(i_1=i_2))$, for some parameter $\beta$.
As for the coloring model, the case $\beta>0$ ($\beta<0$) is called ferromagnetic (anti-ferromagnetic) Ising model. Parameter $h$
is called the external magnetization. It is more common to consider $\{-1,1\}$ as opposed to $\{0,1\}$ as a spin value space,
in which case we can simply write $J(i_1,i_2)=\exp(\beta i_1 i_2)$.
When $h=1$, it is easy to see that the Ising model is equivalent to a special case of the Potts model when $q=2$, by multiplying each $H(\sigma)$
by a constant factor.

\vspace{.2in}
\noindent
\textbf{Viana-Bray model.}
$q=2$. $h(1)=h>0$, for some constant value $h$, and $h(0)=1$. A random variable $I$, which is symmetric around zero and has bounded support,
and parameter $\beta>0$ are fixed. Let
$J(i_1,\ldots,i_K)=\exp\left(\beta I\prod_{1\le l\le K}(\mb{1}(i_l=1)-\mb{1}(i_l=0))\right)$. A more common approach is to fix the set of colors
to be $\{-1,1\}$, in which case $J(i_1,\ldots,i_K)=\exp\left(\beta I\prod_{1\le l\le K}i_l\right)$. The assumption of bounded support is not standard, but is adopted
in this paper for convenience.

\vspace{.2in}
\noindent
\textbf{XOR Model.} Consider the Viana-Bray model with $h=1$ and $I$ taking values $1$ and $-1$ with equal probability.
For convenience let us assume that the spin values are $-1$ and $1$ as opposed to $0$ and $1$.
When $I=1$, the potential $J(i_1,\ldots,i_K)$ takes value $\exp(\beta)$ if an only if an even number of $i_k, 1\le k\le K$ is $-1$,
and otherwise it takes value $\exp(-\beta)$. The situation is reversed when $I=-1$. A more common definition of the XOR model
is as follows: $J(i_1,\ldots,i_K)=(I+i_1\cdots i_K )\mod(2)$. Namely $J=0$ when the parity of the
product $i_1\cdots i_K$ coincides with the parity of $I$, and $J=1$ otherwise.
Thus the model involves hard-core interaction (namely allows $J$ to be zero).
In this case the
corresponding partition function $Z(\G)$ is the number of valid assignments, namely assignments such that every edge potential
value is equal to unity. This version
of the model is outside of the scope of this paper, but we do notice that we obtain it from our version
by defining $J(i_1,\ldots,i_K)=\exp\left(-\beta+\beta I\prod_{1\le l\le K}i_l\right)$ instead, and sending $\beta$ to $+\infty$.
The conditions for existence of valid assignments and the number of valid assignments has been a subject of study
on its own~\cite{MezardMontanariBook},\cite{AbbeMontanari},\cite{IbrahimiKanoriaKraningMontanari}.

\vspace{.2in}
\noindent
\textbf{Random K-SAT model.}
This is our first example for which
considering order of nodes in the edges $e=(u^e_1,\ldots,u^e_K)$ is relevant. We set  $q=2$.
$h_u=h\equiv 1$ for all $u$. For every edge $e$ a vector $(i_1^*,\ldots,i_K^*)$ is selected uniformly at random from $\{0,1\}^K$. We define
\begin{align*}
J_e(i_1,\ldots,i_K)=\left\{
                      \begin{array}{ll}
                        1, & \hbox{$(i_1,\ldots,i_K)\ne (i_1^*,\ldots,i_K^*)$ ;} \\
                        \exp(-\beta), & \hbox{$(i_1,\ldots,i_K)=(i_1^*,\ldots,i_K^*)$.}
                      \end{array}
                    \right.
\end{align*}
Again a more common version of this model is to defined $J_e(i_1,\ldots,i_K)$ to be $0$ when $(i_1,\ldots,i_K)=(i_1^*,\ldots,i_K^*)$
and $1$ otherwise. $Z(\G)$ is the number of satisfiable assignments.
As for the Viana-Bray model we can think of this model as a limit as $\beta\rightarrow\infty$.

\subsection{Continuous spin values}
We now generalize the notion of graph homomorphisms to the case of real valued colors (spins).
This is achieved by  making
$h$ and $J$ random measurable functions with real valued inputs.
Specifically,
assume that we have random i.i.d. (Lebesgue) measurable
function $h_u:\R\rightarrow \R_+, ~u\in V(\G)$, and
random i.i.d. (Lebesgue) measurable functions $J_e:\R^K\rightarrow \R_+, ~e\in E(\G)$. The probability measures
corresponding to the randomness in choices of $h_u$ and $J_e$
are again denoted by $\nu_h$ and $\nu_J$, respectively.
Every homomorphism, which now is assumed to be any map
$\sigma:V(\G)\rightarrow \R$, is associated with a weight defined as
\begin{align*}
H(\sigma)\triangleq \prod_{u\in V(\G)}h_u(\sigma(u))\prod_{e\in E(\G)} J_{e}(\sigma(u^e_1),\ldots,\sigma(u^e_K)),
\end{align*}
and the associated partition function is defined as the following integral taken in the Lebesgue sense:
\begin{align*}
Z(\G)\triangleq \int H(\sigma)d\sigma.
\end{align*}
A more conventional way to write the partition function is to think of  $(\sigma(u), u\in V(\G))$ as an $N=|V(\G)|$-dimensional real vector
$(x_u, u\in V(\G))=(x_1,\ldots,x_N)$, thus defining the associated partition function
\begin{align*}
Z(\G)=\int_{x=(x_1,\ldots,x_N)\in \R^N} \prod_{1\le u\le N}h_u(x_u)\prod_{e\in E(\G)} J_{e}(x_{u^e_1},\ldots,x_{u^e_K})dx.
\end{align*}
We adopt this notational convention from this point on.
It is  simple to see that the discrete case is a special case of the continuous spin value model. Indeed given
a discrete model corresponding to some $K,q$ and realizations of $h_u,J_e$,
define for every $x\in \R$
\begin{align}\label{eq:DiscreteContinuous}
h_u(x)=\left\{
         \begin{array}{ll}
           h_u(i), & \hbox{if $x\in [i,i+1)$, for  $i=0,1,\ldots,q-1$;} \\
           0, & \hbox{otherwise.}
         \end{array}
       \right.
\end{align}
Similarly, $J_e(x_1,\ldots,x_K)=J_e(i_1,\ldots,i_K)$ if $x_{i_l}\in [i_l,i_l+1), l=1,\ldots,K,$ and $J_e=0$ otherwise.
Then $\int H(x)dx=\sum_{\sigma}H(\sigma)$, with the appropriate meaning of $x$ and $\sigma$ in two expressions.

\subsection{Right-convergence  of  graph sequences}
We will be interested primarily  in the case of sequences of sparse graphs $\G$, which for the purposes of this paper we define
as a sequence of graphs $(\G_N, N\ge 1)$ such that $\sup_N |E(\G_N)|/|V(\G_N)|<\infty$. Namely, the number of edges grows  at most linearly
in the number of nodes. For simplicity we assume from now on that $|V(\G_N)|=N$ and $V(\G_N)=\{1,\ldots,N\}$.
The right-convergence of graph sequences concerns  the existence of the limit of the normalized log-partition function
$N^{-1}\log Z(\G_N)$. Here $\log Z(\G_N)$ is defined to be $-\infty$ in the case $Z(\G_N)=0$.
The following definition applies to both the discrete and continuous spin value cases.
\begin{Defi}
Given probability measures $\nu_h,\nu_J$, a sequence of graphs $(\G_N)$ is defined to be right-converging with respect to $\nu_h$ and $\nu_J$
if the sequence $\lim_{N\rightarrow\infty}N^{-1}\log Z(\G_N)$ converges in distribution to some random variable $Z$ in probability.
The sequence of graphs is defined to be right-converging if it is right-converging
with respect to every $\nu_h,\nu_J$.
\end{Defi}
In many of the interesting cases, included the case considered in this paper, the random variable $N^{-1}\log Z(\G_N)$ will be concentrated
around its mean $\E[N^{-1}\log Z(\G_N)]$, in which case the right-convergence is equivalent to the existence of a deterministic quantity $z$
such that
\begin{align*}
\lim_{\epsilon\downarrow 0}\lim_{N\rightarrow\infty}\pr\left(|N^{-1}\log Z(\G_N)-z|>\epsilon\right)=0.
\end{align*}
It is easy to construct examples of trivially converging graph sequences. For example suppose $E(\G_N)=\emptyset$.  Then
\begin{align*}
Z(\G_N)=\prod_u \int_{x\in\R}h_u(x)dx,
\end{align*}
implying that $N^{-1}\log Z(\G_N)$ is an average of i.i.d. sequence of random variables distributed as $\log \int_{x\in\R}h_u(x)dx$.
The limit
equals $\E\log \int h_u(x)dx$, if this expectation is finite, or is some sort of a stable law otherwise.
In general, it is easy to see that if $G_N$ is a disjoint
union of $N/k$ identical graphs with $k$ nodes, where $k$ is constant, then $G_N$ is right-converging.

The notion of right-convergence comes in contrast to the notion of left-convergence~\cite{benjamini_schramm},
which is roughly defined as convergence of constant depth neighborhoods of randomly uniformly chosen nodes in $\G_N$. We do not provide a formal
definition of left-convergence since we will not be working with this notion in this paper. It is known that right-convergence implies
left-convergence~\cite{BorgsChayesKahnLovasz}, but very simple examples exist showing that converse is not true (see again~\cite{BorgsChayesKahnLovasz}).

In this paper we are interested in right-convergence of sequences of \emph{random} graphs. A sequence of random graphs $\G_N, N\ge 1$
is defined to be right-convergent, if sequence of graph  is right-converging in probability
with respect to the randomness associated both with the graph and potentials $h_u,J_e$.
In this paper we
consider exclusively the following sequence of sparse random graphs on $N$ nodes, also known as (sparse) \ER graph,
denoted by $\G(N,c)$.
A constant $c>0$ is fixed. For each $j=1,2,\ldots,\lfloor cN\rfloor$ the $j$-th edge is an ordered $K$ tuple $e_j=(u^{e_j}_1,\ldots,u^{e_j}_K)$  chosen
uniformly at random from the set of nodes $1,2,\ldots,N$, repetition allowed. It is not hard to show that
the graph sequence $\G(N,c), N\ge 1$ is right-converging when $c<1/K$, as in this case the graph breaks down into a disjoint union
of linearly many graphs with a constant average size~\cite{BollobasBook},\cite{JansonBook}. A far more interesting case is when $c>1/K$, as in this
case a giant (linear) size connected component exists and understanding the limit of the log-partition function in this regime is far from trivial.

\section{Conjectures and the main result}\label{section:ConjecturesResults}
\subsection{Assumptions and main results}\label{subsection:AssumptionsResults}
Our main goal is establishing the conditions under which the sequence of \ER graphs is right-converging. Since there is no
particular reason as to why there should exist any measures $\nu_h,\nu_J$ and any $c>0$ such that $\G(N,c)$ would not be right-converging with respect
to $\nu_h,\nu_J$, and since at the very least no counterexamples are known to the day, the following conjecture seems plausible.

\begin{conj}\label{conjecture:GeneralRight-Convergence}
For every $c>0$, the random graph sequence $\G(N,c)$ is right-converging.
\end{conj}

Now we turn to the Talagrand's Problem 6.7.2 in~\cite{TalagrandBook}, mentioned above.
Framed in our terminology, it corresponds to the special case of \ER graph convergence
when $J$ takes only values $0$ and $1$ and $h$ corresponds to a Gaussian
Kernel. There are compelling reasons to consider even this special case and the details can be found in the aforementioned book. Furthermore,
it is over interest to consider even more restricted case when the sets defined by $J=1$ condition are convex. Since the motivation behind the latter
assumption is beyond the scope of the present paper, we will not focus on it here.

The precise statement of Talagrand's conjecture is as follows.
\begin{conj}[Research Problem 6.7.2 in~\cite{TalagrandBook}]\label{conjecture:Talagrand}
For every $c>0$, the random graph sequence $\G(N,c)$ is right-converging when $h(x)=\exp(-x^2)$ and
$J\in \{0,1\}$.
\end{conj}
Despite the Gaussian distribution suggested by the kernel $h(x)=\exp(-x^2)$, one should be aware of the fact that we are dealing here with the case of
a deterministic node potential. In fact, all of our previous examples also correspond to deterministic node potentials as well. The author is not
aware of  interesting models with genuinely random node potentials studied in the past, but since the techniques of the present paper easily
extend to the case of random node potentials, they are allowed in our model.

We now turn to assumptions needed for the statement and the proof of our main result.

\begin{Assumption}\label{assumption:SoftStates}
There exist values $\kappa>0, \rho_{\max}\ge \rho_{\min}>0$, $0<J_{\max}\le \rho_{\max}$ and $\Omega_h\subset \R$
such that
\begin{enumerate}
\item $[0,\kappa]\subset \Omega_h$, and, almost surely with respect to $\nu_h$, $h(x)=0$ for all $x\notin\Omega_h$.
Namely, the support of $h$ lies in $\Omega_h$ almost surely.
Furthermore,
\begin{align}
\rho_{\min}\le \int_{x\in [0,\kappa]} h(x)dx\le \int_\R h(x)dx\le \rho_{\max} \label{eq:hLowerUpper}
\end{align}
almost surely.
\item
$\sup_{x\in \R}J(x)\le J_{\max}$ almost surely.
\item For every $i=1,2,\ldots,K$ almost surely
\begin{align*}
\{x\in\R^K: x_i\in [0,\kappa), x_k\in\Omega_h, 1\le k\le K\}\subset \{x\in\R^K: J(x)\ge \rho_{\min}\}.
\end{align*}

\end{enumerate}
\end{Assumption}
The last assumption means that there is a positive measure set of spin values which have a positive interaction with
every other spin value choices within $\Omega_h$.
In the statistical physics terminology this means that continuous spin values in $[0,\kappa)$ have a positive ("soft")
interaction with all other spin values.

Let us now verify that Assumption~\ref{assumption:SoftStates} holds for all models discussed in Subsection~\ref{subsection:examples}.
For all of these models we take $\Omega_h=[0,q-1]$. Namely, it is the range of continuous representation of the discrete set of colors $0,1,\ldots,q-1$.
For the case of Independent Set model the parts 1) and 2) are verified by taking $\kappa=1,\rho_{\min}=\min(\lambda,1),
\rho_{\max}=1+\lambda$, and $J_{\max}=1$. The assumption 3) concerning the soft core interaction is verified
as well since $J$ takes value $1$ as long as at least one of the arguments of $J$ belongs to $[0,1]$.

The Assumption~\ref{assumption:SoftStates} is verified in a  straightforward way for Partial Coloring and Ising models, as in this case
$J\ge \exp(-\beta)>0$ for any choice of spin values. We set $\kappa=q-1$ so that all values in the domain $\Omega$ "qualify" for soft-core interaction.
We set $\rho_{\max}=\max(\max(1,h)q), J_{\max}=\exp(|\beta|)$, and $\rho_{\min}=\min(\min(1,h),\exp(-|\beta|)$.
For the Viana-Bray model we set
$J_{\max}=\rho_{\max}=\max(h,\exp(\beta c_I))$, where $[-c_I,c_I]$ is the support of $I$.
The remaining parts of the assumptions for the Viana-Bray and verification of the assumptions for XOR is similar.
For the K-SAT model we  set $J_{\max}=1$, and $\kappa=\rho_{\max}=2, \rho_{\min}=\exp(-\beta)$. It is easy to check that
Assumption~\ref{assumption:SoftStates} is verified.

In general it is easy to  see that for discrete deterministic models,
the only non-trivial part of Assumption~\ref{assumption:SoftStates} is the existence of a state with soft interactions.
Namely, the existence of a state $q_0\in \{0,1,\ldots,q-1\}$ such that for every $k=1,\ldots,K$ and every $i_1,\ldots,i_K\in \{0,1,\ldots,q-1\}$ such
that $i_k=q_0$, we have
$J(i_1,\ldots,i_K)>0$. In this case we can take
\begin{align}
J_{\max}&=\max_{0\le i_1,\ldots,i_K\le q-1}J(i_1,\ldots i_K), \label{eq:rhoMaxK2Jdeterministic1}\\
\rho_{\max}&=\max\left(J_{\max},q\max_i h(i)\right). \label{eq:rhoMaxK2Jdeterministic2}\\
\rho_{\min}&=\min_{1\le k\le K}\min_{i_1,\ldots,i_K: i_k=q_0}J(i_1,\ldots,i_K). \label{eq:rhoMinK2Jdeterministic2}
\end{align}

We now turn to our next key assumption regarding the convexity property of the edge potentials $J_e$.
For this purpose we need to resort to the notions of multidimensional arrays and their tensor products.
An $n$-dimensional array $A$ of $K$-th order  is ordered set of real values of the form $A=(a_{i_1,\ldots,i_K}, 1\le i_1,\ldots,i_K\le n)$.
For example, arrays of order $2$ are just $n$ by $n$ matrices. A tensor product of two arrays $A=(a_{i_1,\ldots,i_K})$ and $B=(b_{i_1,\ldots,i_K})$,
of the same dimension $n$ and order $K$ is an $n^2$ dimensional array of order $K$ denoted by
$A\otimes B$, where for each $(i_{1,1},i_{2,1}),\ldots,(i_{1,K},i_{2,K})$ the corresponding entry of $A\otimes B$ is
 $a_{i_{1,1},\ldots,i_{1,K}}b_{i_{2,1},\ldots,i_{2,K}}$.
Given a convex set $S\subset \R^n$, an
array $A$ is defined to be convex over $S$  if the following multilinear form defined on $y=(y_1,\ldots,y_n)\in S$ is convex:
\begin{align}\label{eq:quadraticfunctional}
(y_1,\ldots,y_n)\rightarrow \sum_{1\le i_1,\ldots,i_K\le n}y_{i_1}y_{i_2}\cdots y_{i_K}a_{i_1,\ldots,i_K}.
\end{align}
When $S=\R$ we simply say the array is convex.
For example, the second order array is convex if and only if it is positive semi-definite.
On the other hand, consider a one-dimensional array of order $K=3$, which is defined by a single number $a>0$.
The corresponding multilinear form is just $y\rightarrow ay^3$
is convex over $\Omega=\R_+$ but is not convex over $\Omega=\R$. This observation will be useful in our analysis of the random K-SAT
problem when $K$ is odd.
We write $\langle y,A\rangle$ for the expression on the right-hand side of (\ref{eq:quadraticfunctional}) for short.
We now state our main result.

\begin{theorem}\label{theorem:MainResult}
Suppose the Assumption~\ref{assumption:SoftStates} holds. Suppose further that
there exists constant $\alpha\ge J_{\max}$ such that for every $x_1,\ldots,x_r\in\Omega_h^n$
the expected tensor product
\begin{align}\label{eq:ExpectedTensorProduct}
\E\bigotimes_{1\le l\le r} A_l
\end{align}
is convex on the set $\R_+^{n^r}$, where $A_l$ is an $n$-dimensional array of order $K$ defined by
\begin{align*}
A_l=\alpha-J=\left(\alpha-J(x^l_{i_1},\ldots,x^l_{i_K}), ~1\le i_1,\ldots,i_K\le n\right).
\end{align*}
Then the graph sequence $\G(N,c)$ is right-converging with respect to $\nu_h,\nu_J$.
\end{theorem}
We note that the tensor product $\bigotimes_{1\le l\le r} A_l$ is an $n^r$-dimensional array of order $K$,
the same copy of random $J$, generated according to $\nu_J$ is used in this tensor product, and the expectation is with respect to the randomness of $J$.
The expectation operator when applied to arrays is understood componentwise.

As we will show in the next section, Theorem~\ref{theorem:MainResult} covers many special cases, some of them already covered in the literature.
In particular, the right-convergence for K-SAT and Viana-Bray models was established in~\cite{FranzLeone}, and the right-convergence for
Independent Set, Ising and Coloring models was established in~\cite{BayatiGamarnikTetali}. The right-convergence for
the K-SAT and XOR model with hard-core interaction
was established in~\cite{AbbeMontanari}, but is not covered by our theorem, since the part 3 of Assumption~\ref{assumption:SoftStates} fails for this model.

\subsection{Examples and special cases}\label{subsection:examples}
Let us verify that the convexity assumption of Theorem~\ref{theorem:MainResult} holds for many examples, including our
examples in Subsection~\ref{subsection:Models}. For most of the example we will be able to verify convexity of the expected
tensor product (\ref{eq:ExpectedTensorProduct}) on the entire $\R^{n^r}$ as opposed to $\R_+^{n^r}$.
In particular, let us now focus  on discrete models with $q$ spin values.
Observe that we can bypass the embedding of the discrete model into a continuous model via (\ref{eq:DiscreteContinuous}).
Furthermore, regarding the special case when $K=2$ and $J$ is deterministic,
it is well-known that the product of positive semi-definite matrices is positive semi-definite.
Thus it suffices to assert the convexity of each individual matrix $\alpha-J(x_i,x_j), x=(x_1,\ldots,x_n)\in \R^n$, rather than their
tensor product. Finally, since $x_i$ take arbitrary values in $0,1,\ldots,q-1$, it suffices to simply verify the convexity
of the $q\times q$ matrix $(\alpha-J(i,j), 0\le i,j\le q-1)$. Recall our earlier  observation that Assumption~\ref{assumption:SoftStates}
holds in this case if there exists $i_0$ such that $J(i_0,j)>0$ for all $j$, namely $\max_i\min_jJ(i,j)>0$. In this case
$J_{\max}$ can be set as (\ref{eq:rhoMaxK2Jdeterministic1}). We  obtain the following result.

\begin{theorem}\label{theorem:JdetermK2}
Suppose $K=2$ and $J$ is a deterministic edge potential such that $\max_i\min_jJ(i,j)>0$ and
the matrix $\left(\alpha-J(i,j), 1\le i,j\le q-1\right)$ is positive semi-definite
for some $\alpha\ge J_{\max}$, where $J_{\max}$ is defined by (\ref{eq:rhoMaxK2Jdeterministic1}).
Then the sequence $G(N,c)$ is right-converging.
\end{theorem}
Let us apply this result to our examples, beginning with the Independent Set model.
The matrix $\alpha-J$ is
\begin{align*}
\left(
  \begin{array}{cc}
    \alpha-1 & \alpha-1 \\
    \alpha-1 & \alpha \\
  \end{array}
\right),
\end{align*}
which is positive semi-definite. For the case of the Partial Coloring model we  obtain
that $\alpha-J$ is a matrix with diagonal entries $\alpha-\exp(-\beta)$ and off-diagonal entries $\alpha-1\le \alpha-\exp(-\beta)$.
This matrix is positive semi-definite since $\beta>0$ (anti-ferromagnetism assumption).
The situation for the anti-ferromagnetic Ising model is the same, since the only difference is the possible presence of the magnetic field.

Before we turn to other examples, we ask the following question: under what conditions the required $\alpha$ can be found for discrete models?
While we do not have the answer for the general case, the answer for the case of symmetric deterministic positive definite matrices is rather simple
(the author wishes to thank L\'{a}szl\'{o} Lov\'{a}sz  for this observation).
\begin{lemma}\label{lemma:RestrictedConvex}
Suppose $J=(J_{i,j}, 1\le i,j\le n)$ is a deterministic, symmetric $n\times n$ matrix. There
exists  $\alpha_0>0$ such that $\alpha-J$ is positive definite for all $\alpha\ge \alpha_0$
if and only if $-J$ is positive definite on the linear subspace
\begin{align}\label{eq:subspaceR0}
R_0\triangleq \{y\in \R^n: \mb{e}_n^Ty=0\}.
\end{align}
Namely $y^T(-J)y>0$ for every nonzero $y\in R_0$.
\end{lemma}

We obtain the following result.
\begin{theorem}\label{theorem:2dimKernel}
Suppose  $J$ is a deterministic matrix such that $-J$ is positive definite on $R_0$ and $\max_i\min_jJ(i,j)>0$. Then
the random graph sequence $\G(N,c)$ is right-converging.
\end{theorem}

\begin{proof}[Proof of Lemma~\ref{lemma:RestrictedConvex}]
Suppose $\alpha-J$ is positive definite for some $\alpha>0$. Fix any non-zero $y\in R_0$ and observe that
\begin{align*}
0\le y^T(\alpha-J)y=\alpha y^T(\mb{e}_n\mb{e}_n^T)y-y^TJy=\alpha(\mb{e}_n^Ty)^2-y^TJy=-y^TJy,
\end{align*}
implying that $-J$ is positive definite on $R_0$.

Conversely, suppose $-J$ is positive definite on $R_0$.
Fix any sequence $\alpha_r\rightarrow\infty$. Let
\begin{align*}
y_r=\arg\min_{y\in\R^n, \|y\|_2=1} y^T(\alpha_r-J)y=\arg\min_{y\in\R^n,\|y\|_2=1}\alpha_r(\mb{e}^T_ny)^2-y^TJy.
\end{align*}
Such $y_r$ clearly exists by the compactness argument. Find $y^*, \|y^*\|_2=1$ such that $y_r$ converges to $y^*$ along some subsequence $r_l, l\ge 1$.
If $\mb{e}^Ty^*=0$ then $y*\in R_0$, giving $-(y^*)^TJy^*>0$. Then we can find large enough $r$, such that
$\min_{y\in\R^n,\|y\|_2=1}\alpha_r(\mb{e}^T_ny)^2-y^TJy>0$ and the assertion is proven.

On the other hand, if $\mb{e}^Ty^*\ne 0$, then we find $r_0$ large enough so that $\alpha_r(\mb{e}^T_ny^*)^2-(y^*)^TJy^*>0$ for all $r\ge r_0$.
Then since $y_{r_l}\rightarrow y$, we can find $r_l$ large enough so that
$\min_{y\in\R^n,\|y\|_2=1}\alpha_{r_l}(\mb{e}^T_ny)^2-y^TJy>0$ and the assertion is established.
\end{proof}
Interestingly, the definiteness condition in Lemma~\ref{lemma:RestrictedConvex} cannot be relaxed to the case when $-J$ is positive semi-definite.
Indeed let
\begin{align*}
J=\left(
  \begin{array}{cc}
    -1 & 0 \\
    0 & 1 \\
  \end{array}
\right).
\end{align*}
Then $y^T(-J)y=-y_1^2+y_2^2$ which is $0$ when $y_1+y_2=0$. Nevertheless, $\alpha-J$ is never positive semi-definite since the determinant
is $-1<0$ (the author wishes to thank Rob Fruend for this counterexample).

Let us now give an example of a continuous spin model for which the product in  (\ref{eq:ExpectedTensorProduct}) is convex.
Such an example can derived as a continuous analogue of the coloring model
with countably infinitely many colors. In particular we let $K=2$. Fix a countable sequence
of mutually disjoint positive Lebesgue measure sets $A_r\subset \R, r\ge 1$, and positive weights $\gamma_r>0, r\ge 1$ such that $\sup\gamma_r<\infty$.
Fix any $\gamma\ge \sup\gamma_r$ and
define
\begin{align}\label{eq:J}
J(x,y)=\gamma-\sum_{r\ge 1}\gamma_r\mb{1}\{x,y\in A_r\}.
\end{align}
Namely, for every vector $(x_i,1\le i\le n)$, the corresponding matrix is
\begin{align*}
A\triangleq \left(\gamma-\sum_{r\ge 1}\gamma_r\mb{1}\{x_i,x_j\in A_r\}, ~1\le i,j\le n\right).
\end{align*}
Then $\gamma-A$ is a positive
semi-definite matrix since for every $y\in \R^n$,
\begin{align*}
y^T(\gamma-A)y=\sum_r\gamma_r(\sum_{i: x_i\in A_r}y_i)^2\ge 0.
\end{align*}
In the special case $\gamma=1,\gamma_r\in \{0,1\}$
we also obtain $J(x,y)\in \{0,1\}$ which conforms with Talagrand's Conjecture~\ref{conjecture:Talagrand}. Thus in the special
case $h(x)=\exp(-x^2)$ corresponding to $\Omega_h=\R$,
in order to satisfy
Assumption~\ref{assumption:SoftStates} corresponding to the existence of soft states, it suffices to have at least one $r$ such that $\gamma_r=0$
and $[0,\kappa)\subset A_r$ for some $\kappa>0$.
(This requirement can be generalized to the case that $A_r$ contains some positive length interval or even positive measure set, by
suitably relabeling the spin values).
\begin{coro}\label{coro:Talagrand2dim}
Conjecture~\ref{conjecture:Talagrand} holds when $J(x,y)=1-\sum_{r\ge 1}\gamma_r\mb{1}\{x,y\in A_r\},~\gamma_r\in \{0,1\}$
and there exists $r_0$ and $\kappa>0$ such that $\gamma_{r_0}=0$ and $[0,\kappa)\subset A_{r_0}$,
where $A_r$ is any countably
infinite collection of mutually disjoint measurable subsets of $\Omega_h$.
\end{coro}
Unfortunately, for the case when $K=2$ potentials $J$ are zero-one valued, the example above  is the only
form which can make $\alpha-J$ positive semi-definite for some $\alpha$. Indeed, suppose $J$ is zero-one valued deterministic
edge potential such that the product (\ref{eq:ExpectedTensorProduct}) is convex for every $x_1,\ldots,x_r\in \R^n$.
Let $A_0$ be the set of $x$ such that the measure of the set $\{y: J(x,y)=0\}$ is non-zero. We claim that
$J(x,x)=0$ for every $x\in A_0$. Indeed, otherwise there is $x'$ such that $J(x,x')=0$. Then for the vector $(x_1,x_2)=(x,x')$
and every $\alpha>1$, the matrix $\alpha-J$ is
\begin{align*}
\left(
  \begin{array}{ccc}
    \alpha-J(x,x) & \alpha-J(x,x')  \\
    \alpha-J(x,x') & \alpha-J(x',x')  \\
  \end{array}
\right)=
\left(
  \begin{array}{ccc}
    \alpha-1 & \alpha  \\
    \alpha & \alpha-J(x',x')  \\
  \end{array}
\right).
\end{align*}
The entry $\alpha-J(x',x')$ is either $\alpha$ or $\alpha-1$. As a result the determinant of the matrix is at most $-\alpha^2<0$, and thus
the matrix is not positive semi-definite, and the claim is established.

If $A_0$ is zero measure set, then we can take $A_1=\Omega_h$ and the assertion is established.
Otherwise define equivalency relation on $A_0$ as follows: $x,y\in A_0$ are equivalent if $J(x,y)=0$. The reflexivity follows from the
observation above that $J(x,x)=0$ for every $x\in A_0$, and symmetry follows from symmetry of $J$. For transitivity, suppose there exists $x_1,x_2,x_3\in A_0$ such that
$J(x_1,x_2)=J(x_2,x_3)=0$, but $J(x_1,x_3)=1$. Then the matrix $(\alpha-J(x_i,x_j), 1\le i,j\le 3)$ is
\begin{align*}
\left(
  \begin{array}{ccc}
    \alpha & \alpha & \alpha-1 \\
    \alpha & \alpha & \alpha \\
    \alpha-1 & \alpha & \alpha \\
  \end{array}
\right),
\end{align*}
which is not positive semi-definite, since the determinant of this matrix is $-\alpha$. This establishes transitivity.
By the equivalency relationship, we have
$A_0=\cup_{r\ge 1} A_r$ where $A_r$ are mutually disjoint positive measure sets, and almost surely $J(x,y)=0$ if and only if $x,y\in A_r$ for some $r\ge 1$.
Thus indeed $J$ can satisfy the assumptions of Theorem~\ref{theorem:MainResult} only when it is of the form (\ref{eq:J}).

Now let us turn to the examples when $K\ge 3$. We begin with the random K-SAT model. Recall that for this model
$J$ takes values $1$ or $\exp(-\beta)$. Assumption~\ref{assumption:SoftStates} holds for this model.
We set $\alpha=1$ and claim that the assumptions of Theorem~\ref{theorem:MainResult} hold as well.  We fix an arbitrary sequence $r$ elements:
$x^l=(x_1^l,\ldots,x_n^l)\in \{0,1\}^n,~l=1,2,\ldots,r $. Fix also a  realization of $J$, and let $z_1^*,\ldots,z_K^*\in \{0,1\}$ be the corresponding
unique binary  assignment such that $J(z_1^*,\ldots,z_K^*)=\exp(-\beta)$, and $J(z_1,\ldots,z_K)=1$ when $(z_1,\ldots,z_K)\ne (z_1^*,\ldots,z_K^*)$.
Consider the corresponding array
$\bigotimes_{1\le l\le r} A_l$, where
\begin{align*}
A_l=\left(1-J(x^l_{i_1},\ldots,x^l_{i_K}), ~1\le i^l_1,\ldots,i^l_K\le n\right).
\end{align*}
Every entry of the tensor product $\bigotimes_{1\le l\le r} A_l$ is conveniently indexed by a sequence\\
$(i^1_1,\ldots,i^r_1), \ldots, (i^1_K,\ldots,i^r_K)$, where $i^l_k, 1\le l\le r, 1\le k\le K$ vary over $1,\ldots,n$.
The corresponding entry is
\begin{align*}
\prod_{1\le l\le r}\left(1-J(x^l_{i^l_1},\ldots,x^l_{i^l_K})\right).
\end{align*}
This entry is non-zero if and only if $(x^l_{i^l_1},\ldots,x^l_{i^l_K})=(z_1^*,\ldots,z_K^*)$, in which case the value is $(1-\exp(-\beta))^r$.
Let $S_0\subset \{1,\ldots,n\}$ be the set of indices such that for every $i\in S_0$, $x^1_i=x^2_i=\cdots=x^r_i$.
Recalling that $J$ is generated uniformly
at random, we see that the expected value of the entry corresponding to $(i^1_1,\ldots,i^r_1),\ldots,(i^1_K,\ldots,i^r_K)$
is $2^{-K}(1-\exp(-\beta))^r$ if for every $k$, $i^1_k,\ldots,i^r_k\in S_0$,  and is zero otherwise. Namely, $\E\bigotimes_{1\le l\le r} A_l$ is rank-1 array corresponding an associated principal sub-array.
In particular, if $\bar S$ is the set of vectors $(j_1,\ldots,j_r)$ such that $j_1,\ldots,j_r\in S_0$, then
for $y\in \R^{n^r}_+$,  the corresponding multilinear form is
\begin{align*}
\langle  y,\E\bigotimes_{1\le l\le r} A_l>=2^{-K}(1-\exp(-\beta))^r\left(\sum_{j_1,\ldots,j_r:(j_1,\ldots,j_r)\in \bar S}y_{j_1\cdots j_r}\right)^K.
\end{align*}
This form is convex since $y^K$ is a convex function on $\R_+$ for every positive integer $K$. We have verified that the K-SAT model
satisfies the assumptions of Theorem~\ref{theorem:MainResult}.

We now turn to the Viana-Bray model for the case when $K$ is even.
We set $\alpha=J_{\max}$. For convenience we will use encoding $-1,1$ for $x$ instead of $0,1$, as it was discussed
when we first introduced the example. In other words $J(x_1,x_2)=\exp(\beta Ix_1x_2)$ for any $x_1,x_2\in \{-1,1\}$.
We will use the same symmetrization trick as in~\cite{FranzLeone} and~\cite{PanchenkoTalagrand}.
For any $x_1,\ldots,x_K\in \{-1,1\}$ observe that
\begin{align*}
\alpha-J(x_1,\ldots,x_K)&=J_{\max}-\exp(\beta I \prod x_i) \\
&=J_{\max}-2^{-1}(\exp(\beta I)+\exp(-\beta I))-2^{-1}(\exp(\beta I)-\exp(-\beta I))\prod x_i \\
&\triangleq f_1(I)-f_2(I)\prod x_i.
\end{align*}
Observe also that by symmetry of the distribution of $I$, and as a result, of $f_2(I)$, for every odd $r$ we have
\begin{align}
\E f_2^r(I)=0. \label{eq:Oddr}
\end{align}
Now we verify the convexity of (\ref{eq:ExpectedTensorProduct}). Fix any sequence $x^l=(x_1,\ldots,x_n)\in \{-1,1\}^n$ for $l=1,2,\ldots,r$,
and consider the corresponding array
$\bigotimes_{1\le l\le r} A_l$, where
\begin{align*}
A_l=\alpha-J&=\left(J_{\max}-\exp(\beta I x^l_{i^l_1}\cdots x^l_{i^l_K}), ~1\le i^l_1,\ldots,i^l_K\le n\right) \\
&=\left(f_1(I)-f_2(I)x^l_{i^l_1}\cdots x^l_{i^l_K}, ~1\le i^l_1,\ldots,i^l_K\le n\right).
\end{align*}
Every entry of the tensor product $\bigotimes_{1\le l\le r} A_l$ is again conveniently indexed by a sequence\\
$(i^1_1,\ldots,i^r_1),\ldots,(i^1_K,\ldots,i^r_K)$, where $i^l_k, 1\le l\le r, 1\le k\le K$ vary over $1,\ldots,n$.
The corresponding entry is then
\begin{align*}
\prod_{1\le l\le r}\left(f_1(I)-f_2(I)x^l_{i^l_1}\cdots x^l_{i^l_K}\right)
&=\sum_{S\subset \{1,\ldots,r\}}\left(f_1^{r-|S|}(I)+(-f_2(I))^{|S|}\prod_{l\in S}x^l_{i^l_1}\cdots x^l_{i^l_K}\right)\\
&=(1+f_1(I))^r+\sum_{S\subset \{1,\ldots,r\}}(-f_2(I))^{|S|}\prod_{l\in S}x^l_{i^l_1}\cdots x^l_{i^l_K},
\end{align*}
where the product over empty set $S$ is assumed to be zero.
Then using the observation (\ref{eq:Oddr}) the expected entry is
\begin{align*}
\E[(1+f_1(I))^r]+\sum_{S\subset \{1,\ldots,r\}, |S|\text{~even~}}\E f_2^{|S|}(I)\prod_{l\in S}x^l_{i^l_1}\cdots x^l_{i^l_K}
\end{align*}
Then for any vector $y=(y_{j_1,\ldots,j_r}, 1\le j_1,\ldots,j_r\le n)\in \R^r$
\begin{align*}
\langle y,\E\bigotimes_{1\le l\le r} A_l\rangle&=\E^K[(1+f_1(I))^r]\left(\sum_{1\le j_1,\ldots,j_r\le n}y_{j_1,\ldots,j_r}\right)^K \\
&+\sum_{S\subset \{1,\ldots,r\}, |S|\text{~even~}}\E f_2^{|S|}(I)
\sum_{1\le i^l_k\le n, l\le r, k\le K}\prod_{1\le k\le K}y_{i_k^1\cdots i^r_k}\prod_{l\in S}x^l_{i^l_1}\cdots x^l_{i^l_K}
\end{align*}
The first summand is a convex function since $K$ is even. The second summand is
\begin{align*}
\sum_{S\subset \{1,\ldots,r\}, |S|\text{~even~}}\E f_2^{|S|}(I)
\left(\sum_{1\le j_1,\ldots,j_r\le n}y_{j_1\cdots j_r}\prod_{l\in S}x^l_{j_l}\right)^K.
\end{align*}
The prefactor $\E f_2^{|S|}(I)$ is non-negative since $|S|$ is restricted to be even. Finally, the remaining term is convex since $K$ is even
and the term inside the power $K$ is a linear form in vector $y$.

\section{Preliminary technical results}\label{section:Concentration}
We first obtain some basic upper and lower bound on the log-partition functions.
\begin{lemma}\label{lemma:LogParitionFinite}
Under the Assumption~\ref{assumption:SoftStates} for every graph $\G$ with $N$ nodes and $M$ edges
\begin{align}
(M+N)\log\rho_{\min}&\le  \log Z(\G)\le (M+N)\log\rho_{\max},
\label{eq:LowerUpperBoundZ1}
\end{align}
almost surely.
As a result $\E\log Z(\G))$ is well defined for every graph $\G$.
\end{lemma}

\begin{proof}
We have
\begin{align*}
Z(\G)&\le \int_{\R^N}\prod_{1\le u\le N} h_u(x_u) \prod_{e\in E(\G)} J_e(x_{u^e_1},\ldots,x_{u^e_K}))dx\\
&\le J_{\max}^{M}\prod_{1\le u\le N}\int_\R h_u(x)dx \\
&\le \rho_{\max}^{M+N}
\end{align*}
and
\begin{align*}
Z(\G)&\ge \int_{x\in [0,\kappa)^N}\prod_u h_u(x_u) \prod_e J_e(x_{u^e_1},\ldots,x_{u^e_K})dx \\
&\ge \rho_{\min}^{M}\prod_u \int_\R h_u(x)dx \\
&\ge \rho_{\min}^{M+N}
\end{align*}
from which (\ref{eq:LowerUpperBoundZ1}) follows.
\end{proof}

We now study the impact of adding/deleting one edge from a given realization of a graph and node and edge potentials.
\begin{lemma}\label{lemma:NodeChange}
Consider any graph $\G$ and realizations of node and edge potentials. Suppose the potential in node
$v\in V(\G)$ is changed from $h_u$ to $\hat h_u$, such that Assumption~\ref{assumption:SoftStates} is still valid.
Denote the resulting instance by $\hat \G$. The following holds almost surely
\begin{align}\label{eq:NodeChange}
|\log Z(\G)-\log Z(\hat\G)|\le 2(1+|\mathcal{N}(u,\G)|)(\log\rho_{\max}-\log\rho_{\min}).
\end{align}
\end{lemma}

\begin{proof}
Let $\G_0$ be the graph obtained from $\G$ after deleting node $v$ and all the edges in $\mathcal{N}(v,\G)$ (together with the node potential of $u$
and edge potentials corresponding to $\mathcal{N}(u,\G)$).
Applying Assumption~\ref{assumption:SoftStates} we have
\begin{align*}
Z(\G)&=\int_{\R^N}\prod_u h_u(x_u)\prod_{e\in E(\G)} J_e(x_{u^e_1},\ldots,x_{u^e_K})dx\\
&\le  \rho_{\max}^{|\mathcal{N}(v,\G)|}\int_\R h_v(x)dx
\int_{x\in\R^{N-1}}\prod_{u\ne v} h_u(x_u)\prod_{e\in E(\G)\setminus\mathcal{N}(v,\G)}J_e(x_{u^e_1},\ldots,x_{u^e_K})dx\\
&\le \rho_{\max}^{1+|\mathcal{N}(v,\G)|}Z(\G_0),
\end{align*}
where $Z(\G_0)$ is the partition function associated with $\G_0$:
\begin{align*}
Z(\G_0)=\int_{x\in\R^{N-1}}\prod_{u\ne v} h_u(x_u)\prod_{e\in E(\G)\setminus\mathcal{N}(v,\G)}J_e(x_{u^e_1},\ldots,x_{u^e_K})dx.
\end{align*}
and where the first inequality follows since $x_v$ is not coupled anymore with $x_u, u\ne v$ through the edge potentials.
On the other hand, applying the third part of Assumption~\ref{assumption:SoftStates}
\begin{align*}
Z(\G)&\ge \int_{x_v\in [0,\kappa), x_u\in \R, u\ne v}\prod_u h_u(x_u)\prod_j J_e(x_{u^e_1},\ldots,x_{u^e_K})dx \\
&\ge \rho_{\min}^{1+|\mathcal{N}(v,\G)|}Z(\G_0).
\end{align*}
We obtain
\begin{align*}
|\log Z(\G)-\log Z(\G_0)|\le (1+|\mathcal{N}(v,\G)|)(\log\rho_{\max}-\log\rho_{\min}).
\end{align*}
Since the same bound holds for $\hat G$, we obtain the required bound.
\end{proof}

\begin{lemma}\label{lemma:EdgeChange}
Consider any graph $\G$ and realizations of node and edge potentials. Suppose an edge $e=(v_1,\ldots,v_K)$ is added to the graph
together with an edge potential $J_e$ satisfying Assumption~\ref{assumption:SoftStates}.
Denote the resulting instance by $\hat \G$.
The following holds almost surely
\begin{align}\label{eq:EdgeChange}
\left|\log Z(\G+e)-\log Z(\G)\right|\le (2K+2|\mathcal{N}(e,E)|+1)\left(\log\rho_{\max}-\log\rho_{\min}\right).
\end{align}
\end{lemma}

\begin{proof}
Consider the graph $\G_0$ obtained from $\G$ by deleting nodes $v_1,\ldots,v_K$  and all the edges in $\mathcal{N}(e,\G)$,
together with their associated node and edge potentials.
From Assumption~\ref{assumption:SoftStates} we obtain
\begin{align*}
Z(\G)&=\int_{\R^N}\prod_u h_u(x_u)\prod_e J_e(x_{u^e_1},\ldots,x_{u^e_K})dx\\
&\le  \rho_{\max}^{K+|\mathcal{N}(e,\G)|}
\int_{x\in\R^{N-K}}\prod_{u\ne v_1,\ldots,v_K} h_u(x_u)\prod_{e\notin \mathcal{N}(e,\G)}J_e(x_{u^e_1},\ldots,x_{u^e_K})dx\\
&=\rho_{\max}^{K+|\mathcal{N}(e,\G)|}Z(\G_0).
\end{align*}
On the other hand, by the third part of Assumption~\ref{assumption:SoftStates}
\begin{align*}
Z(\G)&\ge \int_{x_{v_1},\ldots,x_{v_K}\in [0,\kappa), x_u\in \R, u\ne v_1,\ldots,v_K}\prod_u h_u(x_u)\prod_e J_e(x_{u^e_1},\ldots,x_{u^e_K})dx \\
&\ge \rho_{\min}^{K+|\mathcal{N}(e,\G)|}Z(\G_0).
\end{align*}
We conclude
\begin{align*}
|\log Z(\G)-\log Z(\G_0)|\le (K+|\mathcal{N}(e,\G)|)\left(\log\rho_{\max}-\log\rho_{\min}\right).
\end{align*}
Observe that a similar bound holds when $\G+e$ replaces $\G$, where we simply replace $|\mathcal{N}(e,\G)|$ with $|\mathcal{N}(e,\G)|+1$ in upper and lower bounds.
Putting the two bounds together, we obtain the result.
\end{proof}

We now establish a result regarding the concentration of $\log Z(\G(N,c))$ around its mean.
\begin{prop}\label{prop:Concentration}
Under Assumption~\ref{assumption:SoftStates} the following concentration bound holds:
\begin{align*}
\lim_{N\rightarrow\infty}\pr\left(\left|N^{-1}\log Z(\G(N,c))-N^{-1}\E[\log Z(\G(N,c))]\right|>\log^3 N/\sqrt{N}\right)=0.
\end{align*}
\end{prop}

\begin{proof}
A standard approach for proving such concentration result is Azuma-Hoeffding inequality which establishes
concentration bound for martingales with bounded increments.
The application of such a technique would be straightforward if we had a deterministic bound on the edge and node degrees $\mathcal{N}(e,\G)$.
Unfortunately, this is not the case as the largest degree in sparse random graphs $\G(N,c)$ is known to grow at nearly a logarithmic rate.
In order to deal with this we first establish a simple bound on the degree which holds with high probability and then apply the martingale concentration bound
to a truncated version of $Z(\G)$.

Thus let us establish the following simple bound on the largest degree of $\G$.
\begin{align}\label{eq:DegreeBound}
\pr\left(\max_{u\le N}|\mathcal{N}(u,\G)|=\log N\right)=N^{-O(\log\log N)}.
\end{align}
The total number of edges not containing node $u$ is $N^K-(N-1)^K$. Thus  the probability that a randomly chosen node contains $u$ is
\begin{align*}
{N^K-(N-1)^K\over N^K}=1-(1-1/N)^K={K\over N}+o(N^{-1}).
\end{align*}
Given an arbitrary $m$, we then obtain
\begin{align*}
\pr\left(|\mathcal{N}(u,\G)|=m\right)&={\cNfloor\choose m}\left({K\over N}+o(N^{-1})\right)^m\left(1-{K\over N}+o(N^{-1})\right)^m\\
&\le {\cNfloor\choose m}\left({K\over N}+o(N^{-1})\right)^m\\
&\le {(cN)^m\over m!}\left({K^m\over N^m}+o(N^{-m})\right)\\
&={(cK)^m\over m!}+o\left({(cK)^m\over m!}\right)
\end{align*}
When $m\ge \log N$, this bound is ${1\over N^{O(\log\log N)}}$. Using union bound, (\ref{eq:DegreeBound}) follows from
$N^2 N^{-O(\log\log N)}=N^{-O(\log\log N)}$, where factor $N$ in $N^2$ is obtained by over summing over nodes, and the
second factor $N$ is obtained by summing over $\log N\le m\le N$.

We now return to the proof of the concentration result.
For every $n=1,2,\ldots,N$, let $\mathcal{F}_n$ be the filtration associated with random variables $h_u, 1\le u\le n$,
all edges $e$ spanned by the nodes $1,\ldots,n$, as well as
their associated random variables  $J_{e}$. Namely,
$\mathcal{F}_n$ is the information revealed by the portion of the graph $\G(N,c)$ associated with the first $n$ nodes,
edges spanned by these nodes,
as well as their associated potentials $h_u, J_{e}$. Then $R_N=\log Z(\G(N,c))$ and
$R_n\triangleq \E[\log Z(\G(N,c))|\mathcal{F}_n], 0\le n\le N$ is a martingale,
where $\mathcal{F}_0$ is assumed to be a trivial filtration and $\E[\log Z(\G(N,c))|\mathcal{F}_0]=\E[\log Z(\G(N,c))]$.
For every $n$, let $D_n$ be the maximum  degree of the nodes $1,\ldots,n$ in the subgraph spanned by $1,\ldots,n$.
Applying Lemmas~\ref{lemma:NodeChange} and~\ref{lemma:EdgeChange} we have $|R_n-R_{n-1}|\le c_1+c_1D_n^2$, for some constant $c_1>0$,
which depends on $K,\rho_{\min},\rho_{\max}$ only. Indeed the conditioning on the node potential $h_n$ of the node $n$,
by Lemma~\ref{lemma:NodeChange} changes the conditioned expectation by at most $c_1+c_2 D_n$ for some $c_1,c_2$.
Also revealing each edge incident to $n$ and spanned by nodes $1,\ldots,n$ changes the conditioned expectation also by at most
$c_1+c_2 D_n$, by Lemma~\ref{lemma:EdgeChange}. Thus the total change is at most $c_1+c_2 D_n^2\le c_3 D_n^2$,
 for some appropriate constant $c_3$, (where we do not bother to rename the constants $c_1,c_2$).

Let $M\le N$ be defined the smallest $n$ such that $D_n>\log n$.
 If no such node exists (which by (\ref{eq:DegreeBound}) occurs with overwhelming probability), then we set $M=N$. Clearly,
$M$ is a stopping time with respect to the filtration $\mathcal{F}_n$, and $\hat R_n\triangleq R_{\min(M,n)}, 0\le n\le N$ is a stopped martingale,
satisfying $|\hat R_n-\hat R_{n-1}|\le c_3\log^2 N$.
Now applying Azuma-Hoeffding inequality, we obtain  for every $x>0$,
\begin{align*}
\pr\left(\left|\hat R_N-\E\hat R_N\right|> x( c_3\log^3 N) \sqrt{N}\right)\le \exp(-x^2/2).
\end{align*}
From this we obtain
\begin{align*}
\pr\left(\left|R_N-\E[R_N]\right|>\sqrt{N}\log^3 N\right)&\le \pr\left(\left|\hat R_N-\E[\hat R_N]\right|>\sqrt{N}\log^3 N, M=N\right)
+\pr(M<N)\\
&\le \pr\left(\left|\hat R_N-\E[\hat R_N]\right|>\sqrt{N}\log^3 N\right)+N^{-O(\log\log N)}\\
&\le \exp\left(-\log^2 N/(2c_3^2)\right)+N^{-O(\log\log N)}\\
&=N^{-O(\log\log N)}.
\end{align*}
We obtain the claimed concentration result:
\begin{align*}
\pr\left(\left|N^{-1}\log Z(\G(N,c))-N^{-1}\E[\log Z(\G(N,c))]\right|>\log^3 N/\sqrt{N}\right)=N^{-O(\log\log N)}.
\end{align*}
\end{proof}

\section{Interpolation scheme and superadditivity}\label{section:Interpolation}
In this section we introduce the interpolation method and use it to finish to prove our main result, Theorem~\ref{theorem:MainResult}.
Given a positive integer $N$,  consider any positive integers $N_1,N_2$ such that $N_1+N_2=N$.
For every  $t=0,1,\ldots,\cNfloor$ we introduce a  random graph denoted by $\G(N,c,t)$ generated as follows.
The graph $\G(N,c,t)$ has $\cNfloor$ edges. Among those,
$t$ edges are generated independently and uniformly at random among all the $N^K$ potential edges on $N$ nodes. Each of the remaining
$\cNfloor-t$ edges is  generated independently and uniformly at random among the $N_1^K$ edges of the complete graph supported by nodes $1,2,\ldots,N_1$,
with probability $N_1/N$, and is  generated independently uniformly at random among the $N_2^K$ edges of the complete graph supported by
nodes $N_1+1,,\ldots,N$, with probability $N_2/N$.
Observe that $\G(N,c,0)=\G(N,c)$ and $\G(N,c,\cNfloor)$ is a disjoint union of two graphs $\G_j, j=1,2$, where $\G_j$ has $N_j$ nodes
and $R_j$ edges chosen uniformly at random from $N_j^K$ edges, where $R_j$ has a binomial distribution with $\cNfloor$ trials and success
probability $N_j/N$. In particular, the expected number of edges in $\G_j$ is $(N_j/N)\cNfloor\in [cN_j-1,cN_j]$.
Every node $u=1,\ldots,N$ of the graph $\G(N,c,t)$ is equipped with the node potential $h_u$ distributed according to $\nu_h$,
and every edge $e=e_1,\ldots,e_{\cNfloor}$ of the graph is equipped with the edge potential $J_e$ distributed according to $\nu_J$,
all choices made independent. Our main technical result leading to Theorem~\ref{theorem:MainResult} is that the expected
log-partition function of $\G(N,c,t)$ is decreasing as a function of $t$:
\begin{prop}\label{prop:interpolation}
Suppose Assumption~\ref{assumption:SoftStates} holds. Then
\begin{align*}
\E[\log Z(\G(N,c,t))]\ge \E[\log Z(\G(N,c,t+1))]
\end{align*}
for every $0\le t\le \cNfloor-1$.
\end{prop}

Before we prove the proposition we use it to prove our main result.
\begin{proof}[Proof of Theorem~\ref{theorem:MainResult}]
As a corollary of Proposition~\ref{prop:interpolation} we obtain
\begin{align}\label{eq:SuperAdd1}
\E[\log Z(\G(N,c))]\ge \E[\log Z(\G_1)]+\E[\log Z(\G_2)],
\end{align}
where $\G_1$ and $\G_2$ are disjoint parts of $\G(N,c,\cNfloor)$ described above. Since the expected number of edges
of $\G_j, j=1,2$ is in the interval $[cN_j-1,cN_j]$, and the number of edges has binomial distribution with $O(N)$ trials,
then we can obtain a graph $\G(N_j,c)$ from $\G_j$ by deleting or adding at most $O(\sqrt{N})$ edges in expectation.
Applying Lemma~\ref{lemma:EdgeChange} this also implies that for $j=1,2$
\begin{align*}
\Big |\E[\log Z(\G_1)]-\E[\log Z(\G(N_j,c))]\Big |\le O(\sqrt{N}).
\end{align*}
Combining with (\ref{eq:SuperAdd1}) this implies that the sequence $\E[\log Z(\G(N,c))]$ satisfies the following near super-additivity property:
\begin{align*}
\E[\log Z(\G(N,c))]\ge \E[\log Z(\G(N_1,c))]+\E[\log Z(\G(N_2,c))]-O(\sqrt{N}).
\end{align*}
It is a classical fact, known as Fekete's Lemma, that super-additive sequences converge to a limit after the normalization. It is rather
straightforward to show that the same applies to nearly super-additive sequences, provided the correction term is $O(N^\alpha)$, with $\alpha<1$
(and $\alpha=1/2$ in our case). The complete proof can be found in~\cite{BayatiGamarnikTetali}.

We conclude that the following limit exists:
\begin{align*}
\lim_{N\rightarrow\infty}{\E[\log Z(\G(N,c))]\over N}.
\end{align*}
Combining with the concentration result of Proposition~\ref{prop:Concentration}, we conclude that the sequence $\G(N,c)$
is right-converging.
\end{proof}

\begin{proof}[Proof of Proposition~\ref{prop:interpolation}]
Fix any $t<\cNfloor$.
Observe that the graph $\G(N,c,t+1)$ can be obtained from $\G(N,c,t)$ by removing from $\G(N,c,t)$ an edge $e$ (together with the
associated edge potential) chosen uniformly at random from all the edges of $\G(N,c,t)$,
and adding an edge $\hat e$ supported by nodes $1,\ldots,N_1$ chosen uniformly at random from $N_1^K$ possibilities, with probability $N_1/N$,
or supported by nodes $N_1+1,\ldots,N$ chosen uniformly at random from $N_2^K$ possibilities, with probability $N_2/N$.
The newly created edge $\hat e$ is equipped with an edge potential $J_{\hat e}$ generated at random using $\nu_J$, independently from all the other randomness
of the graph.
In this edge removing and edge adding procedure we keep the node potentials intact. Similarly, we keep edge potentials intact for all edges other
than the removed and the added one.
Let $\G_0$ be the realization of the graph obtained after removing edge $e$, but before adding $\hat e$. We assume that $\G_0$
encodes the node/edge potentials as well.
Our proposition will follow from the following inequality which we claim holds for every $\G_0$:
\begin{align}\label{eq:basicInequality}
\E[\log Z(\G(N,c,t))|\G_0]-\log Z(\G_0)\ge \E[\log Z(\G(N,c,t+1))|\G_0]-\log Z(\G_0).
\end{align}
Note by Lemma~\ref{lemma:LogParitionFinite} that $\log Z(G_0)$ as well as both  expectations  are finite, and thus the proposition
indeed follows from (\ref{eq:basicInequality}).
Let $e=(v_1,\ldots,v_K)$ and let $J$ be the corresponding edge potential. Similarly, let $\hat e=(\hat v_1,\ldots,\hat v_K)$
and let $\hat J$ be the corresponding edge potential.
Notice that the randomness of $e,\hat e,J$ and $\hat J$ are the only sources of randomness in the expectations in (\ref{eq:basicInequality}).

Considering $\alpha\ge J_{\max}$ such that (\ref{eq:ExpectedTensorProduct})  is convex, we use Taylor expansion
\begin{align*}
\log x-\log x_0=-\sum_{r\ge 1} r^{-1}(x_0-x)^r x_0^{-r}
\end{align*}
around $x_0=\alpha Z(\G_0)$, we obtain
\begin{align*}
\E[\log &Z(\G(N,c,t))|\G_0,J]-\log Z(\G_0) \\
&=\log \alpha-\sum_{r\ge 1}r^{-1}\alpha^{-r}Z^{-r}(\G_0)\E\left[\left(\alpha Z(\G_0)-Z(\G(N,c,t))\right)^r\right].
\end{align*}
Before we proceed, we need to justify the interchange of infinite summation and expectation.
First observe that $\alpha Z(\G_0)\ge Z(\G(N,c,t))$, since adding an edge can increase the partition function by at most $J_{\max}\le \alpha$
multiplicative factor. (Bound in Lemma~\ref{lemma:EdgeChange} is cruder since we needed it to be two sided). Then the interchange of limits
is justified by the Monotone Convergence Theorem.

With a similar expression for $Z(\G(N,c,t+1))$, we obtain that it suffices to show
\begin{align}\label{eq:leftright}
\E\left[\left(\alpha Z(\G_0)-Z(\G(N,c,t))\right)^r\right]\le \E\left[\left(\alpha Z(\G_0)-Z(\G(N,c,t+1))\right)^r\right].
\end{align}
We begin with the expression on the left and expand it as
\begin{align*}
&\E\left[\left(\alpha Z(\G_0)-Z(\G(N,c,t))\right)^r\right]\\
&=\E\left(\int_{\R^N} \left(\alpha-J(x_{v_1},\ldots,x_{v_K})\right)\prod_{u}h_u(x_{u})
\prod_{e\in E(\G_0)}J_e(x_{u^e_1},\ldots,x_{u^e_K})dx\right)^r\\
&=\E\int_{x^1,\ldots,x^r\in \R^N} \prod_{1\le l\le r}\left(\alpha-J(x^l_{v_1},\ldots,x^l_{v_K})\right)\prod_{u}h_u(x^l_{u})\\
&\times \prod_{e\in E(\G_0)}J_e(x^l_{u^e_1}),\ldots,x^l_{u^e_K})dx^1\cdots dx^r\\
&={N^{-K}}\sum_{1\le v_1,\ldots v_K\le N}
\int_{x^1,\ldots,x^r\in \R^N} \E\prod_{1\le l\le r}\left(\alpha-J(x^l_{v_1},\ldots,x^l_{v_K})\right)\prod_{u}h_u(x^l_{u})\\
&\times \prod_{e\in E(\G_0)}J_e(x^l_{u^e_1},\ldots,x^l_{u^e_K})dx^1\cdots dx^r.
\end{align*}
Here we note that the expectation in the last term is with respect to the randomness of $J$ only.
Given $x^1,\ldots,x^r$, we focus on
\begin{align}\label{eq:eAe}
{N^{-K}}\sum_{1\le v_1,\ldots v_K\le N}\E\prod_{1\le l\le r}\left(\alpha-J(x^l_{v_1},\ldots,x^l_{v_K})\right).
\end{align}
For each $l=1,\ldots,r$, consider the $K$-th order $N$-dimensional array
\begin{align*}
A_l=\left(\alpha-J(x^l_{v_1},\ldots,x^l_{v_K}), 1\le v_1,\ldots,v_K\le N\right).
\end{align*}
Also consider $N^r$-dimensional vector $e^{N,r}$ defined as follows: for every $1\le i_1,\ldots,i_r\le N$,
\begin{align*}
e^{N,r}_{i_1,\ldots,i_r}=\left\{
                           \begin{array}{ll}
                             N^{-1}, & \hbox{$i_1=i_2=\cdots=i_r$;} \\
                             0, & \hbox{otherwise.}
                           \end{array}
                         \right.
\end{align*}
Now observe that (\ref{eq:eAe}) is
\begin{align*}
\E\left[\langle e^{N,r},\bigotimes_{1\le l\le r}A_l\rangle\right].
\end{align*}

Now consider the right-hand size of (\ref{eq:leftright}). For convenience, denote the set of nodes $1,\ldots,N_1$ by $[N_1]$,
and the set of nodes $N_1+1,\ldots,N$ by $[N_2]$.
Using the same expansion, but keeping in mind that we have $\hat e$ in place of $e$, we obtain
\begin{align*}
&\sum_{j=1,2}{N_j\over N}{N_j^{-K}}\sum_{1\le v_1,\ldots v_K\in [N_j]}
\int_{x^1,\ldots,x^r\in \R^N} \E\prod_{1\le l\le r}\left(\alpha-J(x^l_{v_1},\ldots,x^l_{v_K})\right)\prod_{u}h_u(x^l_{u})\\
&\times \prod_{e\in E(\G_0)}J_e(x^l_{u^e_1},\ldots,x^l_{u^e_K})dx^1\cdots dx^r.
\end{align*}
Given $x^1,\ldots,x^r$, we now focus on
\begin{align}\label{eq:eAetilde}
\sum_{j=1,2}{N_j\over N}{N_j^{-K}}\sum_{1\le v_1,\ldots v_K\in [N_j]}\E\prod_{1\le l\le r}\left(\alpha-J(x^l_{v_1},\ldots,x^l_{v_K})\right).
\end{align}
For each $j=1,2$ consider $N^r$-dimensional vector $e^{N,r,j}$ defined as follows: for every $1\le i_1,\ldots,i_r\le N$,
\begin{align*}
e^{N,r,j}_{i_1,\ldots,i_r}=\left\{
                           \begin{array}{ll}
                             N_j^{-1}, & \hbox{if $i_1=i_2=\cdots=i_r\in [N_j]$;} \\
                             0, & \hbox{otherwise.}
                           \end{array}
                         \right.
\end{align*}
Now observe that (\ref{eq:eAe}) is
\begin{align*}
\sum_{j=1,2}{N_j\over N}\E\left[\langle e^{N,r,j},\bigotimes_{1\le l\le r}A_l\rangle\right],
\end{align*}
where $A_l, 1\le l\le r$ are defined as above. By the assumption of convexity of the expected tensor product $\E[\bigotimes_{1\le l\le r}A_l]$,
which is (\ref{eq:ExpectedTensorProduct}),
we obtain
\begin{align*}
\sum_{j=1,2}{N_j\over N}\E\left[\langle e^{N,r,j},\bigotimes_{1\le l\le r}A_l\rangle\right]\ge \E\left[\langle \sum_{j=1,2}{N_j\over N}e^{N,r,j},\bigotimes_{1\le l\le r}A_l\rangle\right].
\end{align*}
Recognizing $\sum_{j=1,2}{N_j\over N}e^{N,r,j}$ as $e^{N,r}$ we obtain the claimed bound (\ref{eq:leftright}). This completes the proof
of Proposition~\ref{prop:interpolation}.
\end{proof}

\section*{Acknowledgements} The author wishes to thank L\'{a}ci Lov\'{a}sz for pointing out the property described in Lemma~\ref{lemma:RestrictedConvex}
and Robert Freund for pointing out the counterexample following this lemma. The author gratefully acknowledges the support by NSF grant CMMI-1031332.
The author wishes to thank Jennifer Chayes and Christian Borgs for many enlightening conversations regarding this work.
Finally the author wishes to thank Microsoft Research Lab at New England where part of this work was conducted.

\newcommand{\etalchar}[1]{$^{#1}$}
\providecommand{\bysame}{\leavevmode\hbox to3em{\hrulefill}\thinspace}
\providecommand{\MR}{\relax\ifhmode\unskip\space\fi MR }
\providecommand{\MRhref}[2]{%
  \href{http://www.ams.org/mathscinet-getitem?mr=#1}{#2}
}
\providecommand{\href}[2]{#2}


\end{document}